\documentclass{amsart}

\usepackage{graphicx}
\usepackage{amssymb}

\newtheorem{theorem}{Theorem}[section]
\newtheorem{proposition}[theorem]{Proposition}

\newtheorem{corollary}[theorem]{Corollary}
\newtheorem{lemma}[theorem]{Lemma}

\theoremstyle{definition}

\numberwithin{equation}{section}

\DeclareMathOperator{\SL}{SL}

\DeclareMathOperator{\G}{\Phi}

\DeclareMathOperator{\comIndex}{c}

\DeclareMathOperator{\Z}{\mathbb{Z}}

\newcommand{\CC}{\mathcal C}

\newcommand{\comG}{\CC}
\newcommand{\comGrowth}{{\bf c}}
\newcommand{\comGrowthSum}{{\bf C}}

\newcommand{\setU}

\def\mf#1{\mathfrak{#1}}

\def\tx#1{{\rm #1}}

\def\R{\mathbb{R}}

\def\Q{\mathbb{Q}}
\def\A{\mathbb{A}}
\def\Z{\mathbb{Z}}
\def\N{\mathbb{N}}
\def\F{\mathbb{F}}

\def\hat{\widehat}

\def\<{\langle}
\def\>{\rangle}

\begin{document}

\title{Commensurability growths of algebraic groups}

\author{Khalid Bou-Rabee}
\address{Department of Mathematics, The City College of New York}
\email{khalid.math@gmail.com}
\thanks{KB supported in part by NSF Grant \#1405609.}

\author{Tasho Kaletha}
\thanks{TK supported in part by NSF Grant \#1161489 and a Sloan Fellowship.}

\address{University of Michigan}
\email{kaletha@umich.edu}

\author{Daniel Studenmund}
\address{University of Notre Dame}
\email{dstudenm@nd.edu}
\thanks{DS supported in part by NSF Grant \#1547292.}


\date{July 9, 2018.}



\begin{abstract}
Fixing a subgroup $\Gamma$ in a group $G$, the full commensurability growth function assigns to each $n$ the cardinality of the set of subgroups $\Delta$ of $G$ with $[\Gamma: \Gamma \cap \Delta][\Delta : \Gamma \cap \Delta] \leq n$. 
For pairs $\Gamma \leq G$, where $G$ is a Chevalley group scheme defined over $\Z$ and $\Gamma$ is an arithmetic lattice in $G$,
we give precise estimates for the full commensurability growth, relating it to subgroup growth and a computable invariant that depends only on $G$.
\end{abstract}

\maketitle

\section*{Introduction}

Two subgroups $\Delta_1$ and $\Delta_2$ of a group $G$ are
\emph{commensurable} if their \emph{commensurability index} 
\[
  \comIndex(\Delta_1, \Delta_2) := [\Delta_1 : \Delta_1 \cap
  \Delta_2][\Delta_2 : \Delta_1 \cap \Delta_2]
\] 
is finite.
For a pair of groups $\Gamma \leq G$, the {\em commensurability growth function} $\N \to \N \cup \{ \infty \}$ assigns to each $n \in \N$ the cardinality
 $$\comGrowth_n(\Gamma,G) := | \{ \Delta \leq G : \comIndex\left(\Gamma, \Delta \right) = n \} |.$$
This function was first systematically studied in \cite{BD18}, where it was used to give regularity results on the structure of arithmetic lattices in a unipotent algebraic group.
Here we tackle the question: How fast does $\comGrowth_n$ grow? Since $\comGrowth_n$ is not necessarily increasing, we look to the \emph{full commensurability growth function} $\N \to \N \cup \{ \infty \}$ assigning to each $n \in \N$ the cardinality
$$
\comGrowthSum_n(\Gamma, G) := \sum_{k=1}^n \comGrowth_k(\Gamma,G).
$$
The above two functions generalizes Lubotzky-Segal's subgroup growth functions $a_n(\Gamma) := \comGrowth_n(\Gamma, \Gamma)$ and $s_n(\Gamma) := \comGrowthSum_n(\Gamma, \Gamma)$, which have been extensively studied \cite{MR1978431}. Indeed, subgroup growth has been estimated for  lattices in semisimple groups \cite{MR2155033, ANS03} and nilpotent groups \cite{MR943928}. In the latter case, many beautiful properties have been shown for associated zeta functions \cite{MR2807855, MR2807857, MR1217350}. Moreover, this theory has lead to the discovery of new groups \cite{LPS96, P04}.

In this article, we begin a systematic study of $\comGrowthSum_n$ for pairs $\Gamma \leq G(\R)$ where $G$ is a linear algebraic group and $\Gamma$ is an arithmetic lattice in $G(\R)$, so that $\Gamma$ is commensurable with $G(\Z)$. 
If $G$ is unipotent, \cite[Lemma 3.1]{BD18} gives that $\comGrowthSum_n(\Gamma, G(\R))$ is always finite. 
In the more general case, it is possible for $\comGrowthSum_n(\Gamma, G(\R))$ to be infinite when $G$ has a compact connected normal $\Q$-subgroup (cf. \cite{BDS18}, where other natural pairs of groups are shown to have infinite commensurability growth values).
For example, the lattice $\SL_2(\Z) \times \{1\} \leq \SL_2(\R)\times \operatorname{SO}(3)$ has infinitely many index 2 supergroups of the form $\SL_2(\Z)\times \left \langle \sigma \right\rangle$ where $\sigma\in \operatorname{SO}(3)$ is a rotation by $\pi$ about an axis in Euclidean space.
For all other semi-simple $G$, the cardinality $\comGrowthSum_n(\Gamma, G(\R))$ is always finite (proved in \S \ref{sec:preliminaries}):

\begin{proposition} \label{prop:finitecomgrowth}
Let $G$ be a semi-simple linear algebraic group with no connected normal $\Q$-subgroup $N \neq \{ e \}$ such that $N(\R)$ is compact. Then for any arithmetic lattice $\Gamma$ in $G$, and any $n \in \N$, $\comGrowthSum_n(\Gamma, G(\R))$ is finite.
\end{proposition}

We will be interested in the asymptotic behavior of $\comGrowthSum_n$. To this end, for two increasing functions $f, g : \N \to \N \cup{\infty}$, we write $f \preceq g$ if there exists $C > 0$ such that $f(n) \leq C g(C n)$ for all $n \in \N$.
We write $f \simeq g$ if $f\preceq g$ and $g \preceq f$, and in this case say $f$ and $g$ \emph{have the same asymptotic growth}. This notion of growth is finer than that in \cite{MR2155033} (for them, $n^2 \asymp n$).
The asymptotic growth of $\comGrowthSum_n(\Gamma, G(\R))$ is called the \emph{commensurability growth} of $G$.
This notation is justified as our next basic result shows that commensurability growth is an invariant of the algebraic group. That is, up to equivalence, the choice of arithmetic subgroup does not change the full commensurability growth, and follows from the following more general result (proved in \S \ref{sec:preliminaries}).

\begin{proposition} \label{prop:cominvariance}
Let $A, B$ be subgroups of $\Delta$ with $\comIndex(A,B) < \infty$.
Then $\comGrowthSum_n(A, \Delta) \simeq \comGrowthSum_n(B, \Delta).$
\end{proposition}

In the course of proving Proposition \ref{prop:cominvariance} we give the space of arithmetic subgroups of $G(\R)$ the structure of a metric space under $d(\Gamma_1,\Gamma_2) := \log( c(\Gamma_1,\Gamma_2))$, which may be of independent interest; see Corollary \ref{lem:commmetric} in \S\ref{sec:preliminaries}.

Armed with this new invariant of an algebraic group, we explore examples starting with the simplest case (see \S \ref{sec:1dimcase} for the proof).

\begin{proposition}
\label{prop:Zcase}
Let $G(\Z) = \Z$ and $G(\R) = \R$.
Then
$\log(\comGrowthSum_n(\Gamma,G(\R))) \simeq \log(n).$
More precisely, 
$$n (\log(n))^{\log(2)} \preceq \comGrowthSum_n(\Gamma,G) \preceq n (\log(n)).$$
\end{proposition}

Proposition \ref{prop:Zcase} demonstrates that $\comGrowthSum_n(G(\Z), G(\R))$ can have different growth from that of $s_n(G(\Z))$. Our main result controls this gap for many simple algebraic groups. For us, a \emph{Chevalley group scheme} is a split simple algebraic group that is simply-connected.

\begin{theorem} \label{thm:main}
Let $G$ be a Chevalley group scheme defined over $\Z$.
Then there exists $M > 0$ such that for any non-uniform arithmetic lattice $\Gamma$ in $G$, 
$$
s_n(\Gamma) \leq \comGrowthSum_n(\Gamma, G(\R)) \preceq n^M s_n(\Gamma).
$$
\end{theorem}


 Essential to our proof is the quantification of a theorem by Borel \cite[Theorem 4]{MR0205999} using ideas and methods from Borel-Prasad \cite{MR1019963} (see Lemma \ref{lem:maximalcount} in \S \ref{sec:semisimplecase}). As an application of Theorem \ref{thm:main} and \cite[Theorem 1]{MR2155033} we obtain:
\begin{corollary} \label{cor:main}
Let $G$ be a Chevalley group scheme defined over $\Z$.
If $G$ has $\R$-rank greater than 1, then for any non-uniform arithmetic lattice $\Gamma$ in $G$, 
$$
\lim_{n \to \infty} \frac{\log(\comGrowthSum_n(\Gamma, G))}{(\log(n))^2/\log\log(n)}
=
\lim_{n \to \infty} \frac{\log(s_n(\Gamma))}{(\log(n))^2/\log\log(n)},
$$
exists and is equal to $\gamma(G)$, a constant depending only on $G$, which is easily computed from the root system of $G$.
\end{corollary}

We were moved to pursue the topic of commensurability growth (first in \cite{BD18}, then in \cite{BDS18} and this paper) after reading the work of N. Avni, S. Lim, and E. Nevo on the related, but different, notion of \emph{commensurator growth} \cite{MR2897732}.

\subsection*{Acknowledgements}
We are grateful to Benson Farb and Gopal Prasad for helpful conversations. We thank Rachel Skipper for helpful comments on an earlier draft.

\section{Basic properties} \label{sec:preliminaries}

For an algebraic group $G$, the space of arithmetic subgroups of $G(\R)$, subgroups commensurable with $G(\Z)$, forms a metric space with discrete topology under the distance function $d(\Gamma_1, \Gamma_2) = \log( c(\Gamma_1, \Gamma_2))$. We prove this by introducing an auxiliary object, the commensurability graph.

Let $G$ be a group and $\mathcal{F}$ a family of subgroups of $G$.
The {\em commensurability graph} associated to such a family is the directed graph with vertex set $\mathcal{F}$ where an edge is drawn from subgroup $H$ to subgroup $K$ if and only if $H\cap K$ is of finite index in $H$ and $K$. 
Each edge between $A, B$ is assigned the weight $\comIndex(A,B)$.
Denote the weighted graph by $\comG_{\mathcal{F}} (G)$, or simply $\comG(G)$ when $\mathcal{F}$ consists of all subgroups of $G$.
Define the {\em length} of a path in $\comG(G)$ to be the product of its edge weights. 
We will consider $\comG(G)$ to be a metric space with the path metric induced by this length.




Let $\gamma$ be a path in some $\CC(G)$ with vertices $\gamma(i)$.
Then we say  $(\gamma(i),
  \gamma(i+1))$ is {\em ascending} if $\gamma(i) \leq \gamma(i+1)$ and
  {\em descending} if $\gamma(i+1)\leq \gamma(i)$.
Any edge $(A,B)$ has the same length as the path $(A, A\cap B), (A\cap B, B)$. It follows that any two vertices in the same component of $\comG(G)$ are connected by a geodesic whose edges are each either ascending or descending.
  
 Our first result of this section shows that any two path-connected points can be connected by a geodesic that is first descending and then ascending.

\begin{lemma} \label{lem:nicegeodesic}
  Let $G$ a group. 
  For any two commensurable subgroups $H, K \leq G$, there is a
  geodesic in $\CC(G)$ supported by the vertices $H, H\cap K,$ and $K$.
  Hence, the distance between $H$ and $K$ in $\CC(G)$ is $c(H,K)$ and the edge $(H,K)$ is a geodesic.
\end{lemma}


\begin{proof}
  Let $\gamma$ be a geodesic from $H$ to $K$ whose edges are each ascending or descending. 
  Suppose there is a consecutive pair of edges $(L,M)$ and $(M,N)$
  in a geodesic such that $L\leq M$ and $N\leq M$. Note that if $h,k$
  are distinct coset representatives for $L\cap N$ in $L$, then
  $hk^{-1} \notin L\cap N$ and therefore $hk^{-1} \notin N$. Therefore
  any collection of distinct coset representatives for $L\cap N$ in
  $L$ is also a collection of distinct coset representative for $N$ in
  $M$, and so $[L : L\cap N] \leq [M:N]$. The same argument shows that
  $[N : N\cap L] \leq [M:L]$. We may therefore replace the consecutive
  pair of edges $(L,M)$ and $(M,N)$ with $(L, L\cap N)$ and
  $(L\cap N, N)$ in any geodesic.

  It follows that there is a geodesic from $H$ to $K$ which is a
  (possibly empty) sequence of descending edges followed by a
  (possibly empty) sequence of ascending edges. 
  Note that, by definition, if $H_1 \geq H_2 \geq H_3$ then the path $(H_1, H_2), (H_2, H_3)$ has the same length as $(H_1, H_3)$. Hence, one can replace the geodesic $\gamma$ with $(H, \Delta), (\Delta, K)$ where $\Delta \leq H \cap K$ without changing its length.
  Since a geodesic has minimal length over all paths, we must have $\Delta = H \cap K$, and so we are done.
\end{proof}

\begin{corollary} \label{lem:commmetric}
    Let $\mathcal{F}$ be a collection of commensurable subgroups of an ambient group $G$. Then $\mathcal{F}$ is a metric space under the function $d(H,K) := \log( c(H,K) )$.
\end{corollary}
\begin{proof}
    The distance is clearly symmetric, and it is not hard to see that $d(H,K)=0$ if and only if $H=K$. Given three subgroups $H,K,L\in \mathcal{F}$, two paths in $\CC(G)$ from $H$ to $K$ are the edge $(H,K)$ and the pair of edges $(H,L), (L,K)$. Lemma \ref{lem:nicegeodesic} implies $c(H,K) \leq c(H,L)c(L,K)$, from which the triangle inequality follows.
\end{proof}

In some cases, being a local geodesic is the same as being a global geodesic.

\begin{corollary} \label{cor:firstloctoglob}
Let $H_1, H_2, \ldots, H_n$ be successive vertices in a path $\gamma$ in some $\CC(G)$.
If $H_i \subset H_{i+1}$ for all $i$, then $\gamma$ is a geodesic.
\end{corollary}

\begin{proof}
Let $\gamma$ be the path given by $H_1, \ldots, H_n$.
Then let $\gamma'$ be the geodesic $H_1, H_n$ guaranteed to exist by Lemma \ref{lem:nicegeodesic}.
Let $p_1 p_2 \cdots p_k$ be the prime factorization of $[H_n : H_1]$.
Then it is straightforward to see that
$$
\text{length of } \gamma = \prod_{i} p_i = \text{ length of }\gamma',
$$
and so $\gamma$ is a geodesic, as desired.
\end{proof}

We apply the above to prove Proposition \ref{prop:cominvariance}:

\begin{lemma} \label{lem:cominvariance}
Let $A, B$ be subgroups of $\Delta$ with $\comIndex(A,B) < \infty$.
Then $\comGrowthSum_n(A, \Delta) \simeq \comGrowthSum_n(B, \Delta).$
\end{lemma}

\begin{proof}
Let $A' \leq \Delta$ be a subgroup with $c(A,A') = n$.
By Lemma \ref{lem:nicegeodesic}, the edge $(A',B)$ is a geodesic in $\comG(G)$, hence we obtain
$$
c(A', B) \leq c(B, A) c(A, A') = c(A, B) n.
$$
This gives
$$
\comGrowthSum_n(A, \Delta) \leq \comGrowthSum_{\comIndex(A,B) n} (B, \Delta).
$$
Similarly,
$$
\comGrowthSum_n(B, \Delta) \leq \comGrowthSum_{\comIndex(A,B) n} (A, \Delta).
$$
\end{proof}

We conclude this section with a proof of Proposition \ref{prop:finitecomgrowth} from the introduction.

\begin{proposition}
Let $G$ be a semi-simple linear algebraic group with no connected normal $\Q$-subgroup $N \neq \{ e \}$ such that $N(\R)$ is compact. Then for any arithmetic lattice $\Gamma$ in $G$, and any $n \in \N$, the cardinality $\comGrowthSum_n(\Gamma, G(\R))$ is finite.
\end{proposition}

\begin{proof}
Let $\Gamma$ be an arithmetic lattice in $G$.
For a fixed natural number $n$, we must show that the set
$$\mathcal{L}_n := \{ \Delta \leq G(\R) : \comIndex(\Delta, \Gamma) = n \}$$
is finite.
Since $\Gamma$ is finitely generated \cite[Theorem 6.12]{MR0147566}, the set
$$
\mathcal{S}_n := \{ \Delta \cap \Gamma: \Delta \in \mathcal{L}_n \},
$$
is finite. It follows that $\Lambda_n := \cap_{A \in \mathcal{S}_n} A$ is a subgroup of finite index in $\Gamma$. By \cite[Theorem 4]{MR0205999}, there are finitely many lattices in $G$ that contain $\Lambda_n$.
It follows that $\mathcal{L}_n$ is finite, and so $\comGrowthSum_n(\Gamma, G(\R))$ is finite, as desired.
\end{proof}

\section{The one-dimensional case} \label{sec:1dimcase}

Notice that if $f(n) \preceq g(n)$, then $\log(f(n)) \preceq \log(g(n))$. Thus, Propositon \ref{prop:Zcase} follows immediately from the following:

\begin{proposition}
\label{prop:ZcaseBody}
Let $G(\Z) = \Z$ and $G(\R) = \R$.
Then
$$n (\log(n))^{\log(2)} \preceq \comGrowthSum_n(\Gamma,G) \preceq n (\log(n)).$$
\end{proposition}

\begin{proof}
We recall from \cite[Proof of Proposition 2.1]{BD18} the formula
$$
\comGrowth_n(\Gamma, G) = 2^{\omega_n},
$$
where $\omega_k$ is the number of distinct primes dividing $k$.
For the lower bound, 
first by \cite[22.10]{MR2445243}
$$
\sum_{k \leq n} \omega_k = n \log\log(n) + B n + o(n),
$$
for a fixed constant $B$.
Then by Jensen's Inequality applied to $\log$,
$$
\log\left(\frac{\sum_{k=1}^n 2^{\omega_k}}{n}\right)
\geq \frac{\log(2) \sum_{k=1}^n \omega_k}n
\geq \log(2) ( \log \log (n) + B + o(n)/n).
$$
Taking the exponential of both sides gives the lower bound:
$$
\sum_{k=1}^n 2^{\omega_k} \succeq n (\log(n))^{\log(2)}.
$$

For the upper bound, let $d(k)$ be the number of natural numbers dividing $k$, then by \cite[Theorem 3.3]{MR0434929},
$$
\sum_{k \leq n} d(k) = n \log(n) + (2\gamma -1) n + O(\sqrt{n}),
$$
where $\gamma$ is Euler's gamma constant.
Since $2^{\omega_k} \leq d(k),$ we get
$$
\comGrowthSum_n(G) \preceq n \log(n),
$$
as desired.



\end{proof}

\section*{The semi-simple case} \label{sec:semisimplecase}

It will be useful for us to control the number of maximal lattices containing a fixed lattice. The following is a quantification of a proof by Borel \cite{MR0205999} and uses ideas and methods from Borel and Prasad \cite{MR1019963}.

\begin{lemma} \label{lem:maximalcount}
Let $G$ be a Chevalley group scheme defined over $\Z$.
Let $H = \ker \left(G(\Z) \to G(\Z/m\Z) \right)$
Then there exits $M_0 > 0$, depending only on $G$, such that the number of maximal lattices containing $H$ is bounded above by $m^{M_0}$. More precisely, $M_0$ can be taken as $3+2d$, where $d$ is the dimension of $G$.
\end{lemma}

\begin{proof}
Given a compact open subgroup $K \subset G(\A_f)$ we obtain a discrete subgroup $\Gamma \subset G(\R)$ by taking the intersection $\Gamma=K \cap G(\Q)$. Then $\Gamma$ is an arithmetic subgroup of $G(\R)$ and the closure of its image under the natural embedding $G(\Q) \to G(\A_f)$ is again $K$ \cite[\S1.3]{MR1019963}. Conversely, if $\Gamma \subset G(\Q)$ is a maximal arithmetic subgroup of $G(\R)$, then its closure $K \subset G(\A_f)$ is a maximal compact open subgroup, hence given as $\prod_p K_p$, where $K_p \subset G(\Q_p)$ is a maximal parahoric subgroup, and is hyperspecial for almost all $p$ \cite[\S1.4]{MR1019963}.

A maximal $\Gamma$ contains $H$ if and only if its closure $K$ contains the closure of $H$, which equals $\tx{ker}(G(\hat \Z) \to G(\hat \Z/m\hat \Z))$. This is in turn equivalent to $K_p$ containing $H_p$, where $H_p=\tx{ker}(G(\Z_p) \to G(\Z_p/p^{k_p}\Z_p))$ and $m=\prod_p p^{k_p}$.

We now focus on a fixed $p$ and count the number of maximal parahoric subgroups $K_p \subset G(\Q_p)$ that contain $H_p$. Let $d=\tx{dim}(G)$. The $G(\Q_p)$-conjugacy classes of maximal parahoric subgroups correspond to the vertices of a $d$-dimensional simplex -- a fixed alcove in a fixed apartment of the Bruhat-Tits building for $G(\Q_p)$. There are thus $d+1$ such conjugacy classes and it is enough to estimate the number of elements in a given conjugacy class that contain $H_p$.

Let $o$ be the hyperspecial vertex of the Bruhat-Tits building corresponding to the $\Z_p$-structure of $G$ coming from its $\Z$-structure. Fix a Borel pair $(T,B)$ in $G$ over $\Q_p$ so that $o$ belongs to the apartment of $T$, and let $x$ be a vertex in the unique alcove of the apartment of $T$ that contains $o$ in its closure. We have the Bruhat decomposition 
$$G(\Q_p) = \bigcup_{\lambda \in X_*(T)} G(\Z_p) \lambda(p) G(\Q_p)_x.$$ 
Here $X_*(T)$ is the lattice of cocharacters of $T$, $G(\Q_p)_x$ is the parahoric subgroup corresponding to $x$, and $G(\Z_p)$ is by definition the parahoric subgroup corresponding to $o$. We thus have the surjective map from $G(\Z_p) \times X_*(T)$ to the conjugacy class of $G(\Q_p)_x$ that sends $(h,\lambda)$ to $h\lambda(p)G(\Q_p)_x\lambda(p)^{-1}h^{-1}$. Since $G(\Z_p)$ normalizes $H_p$, the group $h\lambda(p)G(\Q_p)_x\lambda(p)^{-1}h^{-1}$ contains $H_p$ if and only if $\lambda(p)^{-1}H_p\lambda(p) \subset G(\Q_p)_x$.

Let $R$ be the root system of $G$ with respect to $T$ and let $(u_\alpha : \G_a \to G)_{\alpha \in R}$ be a parameterization of the root subgroups defined over $\Z_p$. For $z \in \Q_p$ we see that $u_\alpha(z) \in \lambda(p)^{-1}H_p\lambda(p)$ if and only if $\tx{ord}_p(z) \geq k-\<\lambda,\alpha\>$ and $u_\alpha(z) \in G(\Q_p)_x$ if and only if $\tx{ord}_p(z) \geq -\<x,\alpha\>$. Comparing these conditions we see that $\lambda(p)^{-1}H_p\lambda(p) \subset G(\Q_p)_x$ holds if and only if $\<\lambda,\alpha\> \leq k+\<\alpha,x\>$ holds for all $\alpha \in R$. Since $x$ and $o$ both belong to the closure of the same alcove we have $\<\alpha,x\> \leq 1$ for all $\alpha \in R$. Every $\lambda \in X_*(T)$ that satisfies $\<\lambda,\alpha\> \leq k+1$ can be written as an integral linear combination $\sum_\alpha a_\alpha \check\varpi_\alpha$ of the fundamental coweights $\check\varpi_\alpha$ corresponding to the set of simple roots $\alpha$, and we must have $|a_\alpha| \leq k+1$. While not sufficient, this condition is certainly necessary, and it allows us to estimate the number of such $\lambda$ by $(2k+3)^d$.

Each $\lambda$ such that $\lambda(p)^{-1}H_p\lambda(p) \subset G(\Q_p)_x$ corresponds to the double coset $G(\Z_p)\lambda(p)G(\Q_p)_x$ and each $G(\Q_p)_x$-cosets in this double coset gives an element of the $G(\Q_p)$-conjugacy class of $G(\Q_p)_x$ that contains $H_p$. For a given such $\lambda$ we now need to estimate the number of $G(\Q_p)_x$-cosets in $G(\Z_p)\lambda(p)G(\Q_p)_x$. The map
\[ G(\Z_p)/(G(\Z_p) \cap \lambda(p)G(\Q_p)_x\lambda(p)^{-1}) \to G(\Z_p)\lambda(p)G(\Q_p)_x/G(\Q_p)_x,\quad h \mapsto h\lambda(p) \]
is a bijection. Since $H_p \subset G(\Z_p) \cap \lambda(p)G(\Q_p)_x\lambda(p)^{-1}$ we can estimate the number of elements of this set by the number of elements of $G(\Z_p)/H_p$. The latter is the cardinality of $G(\Z_p/p^k\Z_p)$. This group has a filtration of length $k$, whose successive quotients are all isomorphic to $\mf{g}(\F_p)$, except for the last one, which is $G(\F_p)$. Here $\mf{g}$ is the Lie algebra of $G$ and  $\#\mf{g}(\F_p)=p^d$. On the other hand, according to \cite[Theorem 25]{MR0466335} we have  $\#G(\F_p)=p^N\prod_i(p^{d_i}-1)$, where $N$ is the number of positive roots, $d_1,\dots,d_l$ are the degrees of the fundamental invariants of the Weyl group, and $l$ is the rank of $G$. We have the relation $N=\sum_i (d_i-1)$. We can therefore estimate 
\[
\#G(\F_p) \leq p^N\prod_i p^{d_i} = p^{N+\sum_i d_i}=p^{2N+l}=p^d.
\]
The cardinality of $G(\Z_p/p^k\Z_p)$ is the product of successive quotients of the aforementioned filtration of length $k$, hence $\#G(\Z_p/p^k\Z_p) \leq p^{kd}$.

In summary, we have estimated the number of maximal compact open subgroups of $G(\Q_p)$ in a fixed $G(\Q_p)$ conjugacy class that contain $H_p$ by $p^{kd}(2k+3)^d$. If $p \geq 5$ then $2k+3 \leq p^k$ for all $k$. We can remove the restriction on $p$ by using the cruder estimate $2k+3 \leq p^{3k}$. Recalling that there are $d+1$ conjugacy classes of maximal compact open subgroups, we estimate the number of all maximal compact open subgroups of $G(\Q_p)$ that contain $H_p$ by $(d+1)p^{(3+d)k}$. Since $d+1 \leq p^d$ we estimate this by $p^{(3+2d)k}$, noting that when $k=0$ there is only one maximal compact subgroup of $G(\Q_p)$ containing $H_p=G(\Z_p)$, namely $G(\Z_p)$ itself. 

Since $m=\prod_p p^{k_p}$, we estimate the number of maximal arithmetic subgroups of $G(\R)$ that contain $H$ by $\prod_p p^{(3+2d)k_p}=m^{(3+2d)}$, as desired.
\end{proof}

\subsection{Proof of Theorem \ref{thm:main}}

\begin{theorem} \label{thm:mainBody}
Let $G$ be a Chevalley group scheme defined over $\Z$.
Then there exists $M > 0$ such that for any non-uniform arithmetic lattice $\Gamma$ in $G$, 
$$
s_n(\Gamma) \leq \comGrowthSum_n(\Gamma, G) \preceq n^M s_n(\Gamma).
$$
\end{theorem}

\begin{proof}
The lower bound is clear.
We proceed to prove the upper bound, by Lemma \ref{lem:cominvariance} we may suppose $\Gamma$ is a maximal lattice containing $G(\Z)$.
Let $\Delta$ be a lattice with $\comIndex(\Gamma, \Delta) \leq n$.
Let $\Delta_1$ be a maximal lattice containing $\Delta$ \cite{MR0205999}.
Let $\mu$ be a Haar measure on $G(\R)$.
There are finitely many conjugacy classes of maximal lattices in $G(\R)$ by \cite{MR2726109}.
Let $A_1, \ldots, A_r$ be representatives from each conjugacy class of maximal lattices. 
Set $C_0 = (\max_{i=1,\ldots, r} \mu(A_i))/\mu(\Gamma)$.
Then it follows that
$$
[\Delta_1 : \Delta \cap \Gamma] 
= \mu(\Delta_1)/\mu(\Delta \cap \Gamma) 
\leq \max_{i=1,\ldots, r} \mu(A_i)/\mu(\Delta \cap \Gamma) = C_0 [\Gamma: \Delta \cap \Gamma].
$$

Since $[\Gamma: \Delta \cap \Gamma] \leq n$, it follows that
$$
[\Delta_1 : \Delta] \leq C_0 n.
$$
Next, let $m$ be the minimal natural number such that $\Delta \cap \Gamma$ contains the kernel of the map $\phi: \Gamma \to G(\Z/m\Z)$.
Then by \cite[Proposition 6.1.2]{MR1978431}, we have that there exists $c > 0$, depending only on $G$, such that $m \leq c [\Gamma: \Gamma \cap \Delta]$.
We conclude that $\Delta$ is a subgroup of index $\leq C_0 n$ for some maximal lattice containing some $m$th principal congruence subgroup, where $m \leq c n$.

By Lemma \ref{lem:maximalcount}, there exists at most $(1 + 2^{M_0} + 3^{M_0} + \cdots (cn)^{M_0})$ possibilities for such maximal lattices. Since there are finitely many isomorphism classes of such maximal lattices (since there are finitely many conjugacy classes), there exists at most 
$$
(1 + 2^{M_0} + 3^{M_0} + \cdots (cn)^{M_0}) s_\Gamma(D n)
$$
possibilities for subgroups $\Delta$ with $\comIndex(\Gamma, \Delta) \leq n$, for some fixed $D > 0$.
Since $(1 + 2^{M_0} + 3^{M_0} + \cdots (cn)^{M_0})$ is a polynomial in $n$, we are done.
\end{proof}

\subsection{Proof of Corollary \ref{cor:main}}
We first remark that \cite[Theorem 1]{MR2155033} holds for $G$, since their proof fails only for $G = {}^6D_4$ \cite[pp. 135]{MR2155033}, which never occurs when $G$ is split.

We proceed with this result in hand. By Theorem \ref{thm:main} we have 
$\comGrowthSum_n(\Gamma, G) \preceq n^M s_n(\Gamma).$ Hence, for some $C > 0$, 
$$
\lim_{n \to \infty} \frac{\log(\comGrowthSum_n(\Gamma, G))}{(\log(n))^2/\log\log(n)}
\leq
\lim_{n \to \infty} \frac{\log(s_{Cn}(\Gamma)) + M \log(n) + M \log(C)}{(\log(n))^2/\log\log(n)}
$$
Since $\frac{\log(n)}{(\log(n))^2/\log\log(n))} \to 0$ as $n \to \infty$, we have 
\begin{equation} \label{eqn:proofcor1}
\lim_{n \to \infty} \frac{\log(\comGrowthSum_n(\Gamma, G))}{(\log(n))^2/\log\log(n)}
\leq
\lim_{n \to \infty} \frac{\log(s_{Cn}(\Gamma))}{(\log(n))^2/\log\log(n)}.
\end{equation}
Since $\lim\limits_{n \to \infty} \frac{(\log(n))^2}{\log\log(n)} \frac{\log\log(Cn)}{(\log(C n))^2} = 1$ and $\lim\limits_{n \to \infty} \frac{\log(s_{n}(\Gamma))}{(\log(n))^2/\log\log(n)}
$ exists by \cite[Theorem 1]{MR2155033}, we have
$$
\lim_{n \to \infty} \frac{\log(s_{Cn}(\Gamma))}{(\log(n))^2/\log\log(n)}
=
\lim_{n \to \infty} \frac{\log(s_{n}(\Gamma))}{(\log(n))^2/\log\log(n)}.
$$
It follows from Inequality (\ref{eqn:proofcor1}) and $s_n(\Gamma) \leq \comGrowthSum_n(\Gamma, G)$ that
$$
\lim_{n \to \infty} \frac{\log(\comGrowthSum_n(\Gamma, G))}{(\log(n))^2/\log\log(n)}
=
\lim_{n \to \infty} \frac{\log(s_n(\Gamma))}{(\log(n))^2/\log\log(n)},
$$
and so Corollary \ref{cor:main} follows from \cite[Theorem 1]{MR2155033}, as desired.

\section{Final remarks}

For $\Gamma$ an arithmetic lattice in a Chevalley group scheme $G$ defined over $\Z$, it would be interesting to improve Theorem \ref{thm:main} to find a number $N$ such that $\comGrowthSum_n(\Gamma, G) \simeq n^N s_n(\Gamma)$.
However, even this would not be enough to determine the asymptotic growth class of $\comGrowthSum_n(\Gamma,G)$; the best known bounds on the long-term behavior of $s_n$ only determine the asymptotic growth of $\log(s_n(\Gamma))$ \cite{MR2155033}. This article shows that $\log(s_n(\Gamma)) \simeq \log(\comGrowthSum_n(\Gamma,G))$.


\bibliography{refs}
\bibliographystyle{amsalpha}

\end{document}